\documentclass[12pt,reqno,a4paper]{amsart}
\usepackage[top=100pt,bottom=60pt,left=96pt,right=92pt]{geometry}
\usepackage{amsmath,amssymb}
\usepackage{tikz}
 
\newtheoremstyle{mystyle}
{}
{}
{\normalfont}
{0pt}
{\bfseries}
{.}
{4pt}
{}
\theoremstyle{mystyle}
\newtheorem{definition}{Definition}[section]

\newtheoremstyle{mystyle2}
{6pt}
{6pt}
{\itshape}
{0pt}
{\bfseries}
{.}
{4pt}
{}
\theoremstyle{mystyle2}

\newtheorem{lemma}[definition]{Lemma}

\begin{document}
\title[Unitary Howe dualities]{Unitary Howe dualities for fermionic and bosonic algebras and related Dirac operators}
\author{Guner Muarem}
\address{Campus Middelheim 
	\\Middelheimlaan 1
	\\M.G.221 
	\\2020 Antwerpen 
	\\	Belgi\"e}
\curraddr{}
\email{guner.muarem@uantwerpen.be}
\thanks{}
\maketitle
\begin{abstract}
	In this paper we use the canonical complex structure $\mathbb{J}$ on $\mathbb{R}^{2n}$ to introduce a twist of the symplectic Dirac operator. As a matter of fact, these operators can be interpreted as the bosonic analogues of the Dirac operators on a Hermitian manifold. Moreover, we prove that the algebra of these symplectic Dirac operators is isomorphic to the Lie algebra $\mathfrak{su}(1,2)$ which leads to the Howe dual pair $(\mathsf{U}(n),\mathfrak{su}(1,2))$.
\end{abstract}
\section{Introduction}\label{sec:intro}
\noindent The CCR (canonical commuting relation) and CAR algebras (canonical anticommuting relation) are fundamental algebras in theoretical physics used for the study of bosons and fermions. From a mathematical viewpoint, these algebras are named the Weyl algebra (or symplectic Clifford algebra) and Clifford algebra. These algebras can be constructed in a very analogous way. The Clifford algebra is constructed on a vector space $V$ equipped with a symmetric bilinear form $B$, whereas the Weyl algebra requires an even dimensional vector space equipped with a skew-symmetric bilinear form (or symplectic form) $\omega$. In both cases, one then constructs the tensor algebra $T(V)$ where an ideal $I(V)$ is divided out. In the orthogonal setting, this is the ideal $I_B(V)$ with elements subject to the relation $\{u,v\}=2B(u,v)$. In the symplectic setting, this is the ideal $I_\omega(V)$ generated by $[u,v]=-\omega(u,v)$. \par There is, however, a fundamental difference: the Clifford algebra is finite-dimensional, whereas the Weyl algebra is infinite-dimensional. For the spinors (orthogonal versus symplectic) the same infinite-dimensional principle holds as for the Clifford algebras.
As a matter of fact, the symplectic spinors are the smooth vectors in the metaplectic representation \cite{Habermann}. Using the generators of the Clifford (resp.\ Weyl) algebra, one can associate a natural first order spin (resp.\ metaplectic) invariant differential operator by contracting the Clifford algebra elements using the bilinear form $B$ (resp. the symplectic form $\omega)$ with derivatives. This gives rise to the Dirac operator $\partial_x=\sum_{k=1}^n e_k\partial_{x_k}$ where $\{e_j,e_k\}=-2\delta_{ij}$ and the symplectic Dirac operator $\sum_{k=1}^n \left(iq_k\partial_{y_k}-\partial_{q_k}\partial_{x_k}\right)$ where $[\partial_{q_j},iq_k]=i\delta_{jk}$ are the Heisenberg relations. \par The theory which studies the solutions of the Dirac operator is known as Clifford analysis and can be seen as a hypercomplex function theory. Moreover, quite some generalisations have occurred in the last two decades. This involves e.g.\ Clifford analysis on superspace and Clifford analysis on (hyper)K\"ahler spaces (see for instance \cite{osp}). It is in the latter framework in which this paper is situated, but then from a symplectic point of view. More precisely, we provide the foundations of what we will call a \textit{hermitian variant of symplectic Clifford analysis}, where we incorporate the additional datum of a compatible complex structure $\mathbb{J}$ on the flat symplectic space $\mathbb{R}^{2n}$. This leads to the study of the symplectic Dirac operator on a K\"ahler manifold as was already initiated in \cite{Habermann,Dol}. However, the underlying invariance symmetry and the algebra generated by these type of operators (and their duals) was never investigated. 
\section{Rudiments of symplectic Clifford analysis}
\noindent Let us consider the canonical symplectic space $\mathbb{R}^{2n}$ with coordinates $({x},{y})$ and the usual symplectic form $\omega_0=\sum_{j=1}^n dx_j\wedge dy_j$ which has the matrix representation $$\Omega_0=\begin{pmatrix}
0	& I_n\\
-I_n&	0
\end{pmatrix}.$$
Recall that the symplectic group $\mathsf{Sp}(2n,\mathbb{R})$ is the group given by invertible linear transformations preserving the non-degenerate skew-symmetric bilinear form from above and is given (in terms of matrices) by 
$$
\mathsf{Sp}(2n,\mathbb{R})=\{M\in \mathsf{GL}(2n,\mathbb{R})\mid M^T\Omega_0M=\Omega_0\}.$$
The group is non-compact and has dimension $2n^2+n$.  Moreover, the corresponding Lie algebra is denoted by $\mathfrak{sp}(2n,\mathbb{R})$. 
The main difference with the orthogonal case, lies in the fact that the metaplectic group (the double cover of the symplectic group) does \textit{not} admit a finite dimensional representation (it is not a matrix group). This is a strong contrast with the spin representation in the orthogonal case.  Moreover, the orthogonal spinors $\mathbb{S}$ are realised as a idempotent left ideal in the Clifford algebra, which is not the case for the symplectic spinors. As mentioned, the symplectic equivalent of the spin representations are infinite dimensional, which means that one needs to work with the theory of unitary representations.
\subsection{The Schwartz space and metaplectic representation}
For further convenience, we fix notation and define the Schwartz space, which plays a crucial role in the construction of the metaplectic representation. On the space $\mathcal{C}^{\infty}(\mathbb{R}^n,\mathbb{C})$ we define (using the multi-index notation) the norm
$$	||f||_{\alpha,\beta}:=\sup_{{q}\in\mathbb{R}^n}|{q}^\alpha(D^{\beta}f)({q})|
$$ for all $\alpha,\beta\in\mathbb{N}^n$. The Schwartz space $\mathcal{S}(\mathbb{R}^n)$ is the subspace of $L^p(\mathbb{R}^n)$ (for $1\leq p\leq \infty$) consisting of rapidly decreasing functions and is given by $$
\mathcal{S}(\mathbb{R}^n,\mathbb{C}):=\{f\in\mathcal{C}^{\infty}(\mathbb{R}^n,\mathbb{C}):||f||_{\alpha,\beta}<\infty \text{ for all } \alpha,\beta\in\mathbb{N}^n\}.
$$
We now describe (following \cite{Habermann}) the infinite-dimensional 
Segal-Shale-Weil representation (also oscillator or metaplectic 
representation) of the metaplectic group. The smooth vectors of 
the unitary representation $\mathfrak{m} : \mathsf{Mp}(2n) \to \mathsf{U}
(L^2(\mathbb{R}^n))$ coincide with the Schwartz space $\mathcal{S}
(\mathbb{R}^n)$ and are a model for the symplectic spinors $\mathbb{S}
^{\infty}$. Due to Stone-Von Neumann theorem the 
representation is unique (up to unitary equivalence).
\subsection{The symplectic Clifford algebra and the related Dirac operator}
	Let $(V,\omega)$ be a symplectic vector space.
The \textit{symplectic Clifford algebra} $\mathsf{Cl}_s(V,\omega)$ is defined as the quotient algebra of the tensor algebra $T(V)$ of $V$, by the two-sided ideal $$\mathcal{I}_{\omega}:=\{v\otimes u-u\otimes v+\omega(v,u) : u,v\in V\}.$$ In other words
$\mathsf{Cl}_s(V,\omega):=T(V)/\mathcal{I}_{\omega}$
is the algebra generated by $V$ in terms of the relation $[v,u]=-\omega(v,u)$, where we have omitted the tensor product symbols.
We refer to the symplectic Clifford algebra on $\mathbb{R}^{2n}$ as the $n$th Weyl algebra $\mathcal{W}_n$ with generators $iq_1,\dots,iq_n,\partial_{q_1},\dots,\partial_{q_n}$ satisfying the commutation relations $[q_j,q_k]=0$ and $[\partial_{q_j},q_k]=\delta_{jk}$.
\par Denote by $\mathcal{F}$ a suitable function space (e.g. the space of polynomials, or smooth funtions). 
\begin{definition}
The {\em symplectic Dirac operator} on $(\mathbb{R}^{2n},\omega_0)$ is the first-order (in the base variables ${x}$ and ${y}$) differential operator acting on a symplectic spinor-valued functions space $\mathcal{F}\otimes \mathbb{S}^{\infty}$ given by 
$
D_s=\sum_{j=1}^n (iq_j\partial_{y_j}-\partial_{q_j}\partial_{x_j}).
$
With respect to the symplectic Fischer inner product (see \cite{sym}), we obtain the dual operator
$
X_s=\sum_{j=1}^n (iq_j{x_j}+\partial_{q_j}{y_j}).
$
\end{definition}
\noindent These operators satisfy the relations:
\begin{align*}
[\mathbb{E}+n,X_s]=X_s,&&
[\mathbb{E}+n,D_s]=-D_s,&&
[D_s,X_s]=-i(\mathbb{E}+n)
\end{align*}
where $\mathbb{E}=\sum_{j=1}^n (x_j\partial_{x_j}+y_j\partial_{y_j})$ is the \textit{Euler operator}.
In other words, the three operators give rise to a copy of the Lie algebra $\mathfrak{sl}(2)$.
\section{The interaction with a complex structure}
\subsection{Definition of the twisted symplectic Dirac operators}
We will now introduce a complex structure $\mathbb{J}$ on the symplectic manifold $(\mathbb{R}^{2n},\omega_0)$ which is compatible with the symplectic form $\omega_0$. This means that $\omega_0(x,\mathbb{J}y)$ defines a Riemannian metric $g$. Otherwise said,  we will be working with the canonical K\"ahler manifold $(\mathbb{R}^{2n},\omega_0,g,\mathbb{J})$. By Darboux's theorem, we obtain, with respect to the canonical symplectic basis $\{e_j\}_{j=1}^{2n}$ the following complex structure $\mathbb{J} = \left(\begin{smallmatrix}
0	&	-I_n \\
I_n	&	0
\end{smallmatrix}\right).$
The action of the complex structure $\mathbb{J}$ on $\mathbb{R}^{2n}$ is given by
\begin{align*}
(x_1,\dots,x_n,y_1,\dots,y_n)\mapsto (y_1,\dots,y_n,-x_1,\dots,-x_n).
\end{align*}
\begin{definition}
	The new differential operators acting on symplectic spinor-valued functions \begin{align*}
\tilde{D}_s &=\sum_{j=1}^niq_j\partial_{x_j}+\partial_{y_j}\partial_{q_j}
&\tilde{X}_s &=\sum_{j=1}^n x_j\partial_{q_j} - iy_jq_j
&\mathbb{E}&=\sum_{j=1}^nx_j\partial_{x_j}+y_j\partial_{y_j}
\end{align*} 
also give rise to a copy of the Lie algebra $\mathfrak{sl}(2)$. We call these first two operators the \textit{twists} of $D_s$ and $X_s$.
\end{definition}
 Both sets of operators, i.e.\ $(D_s,X_s)$ and $(\tilde{D}_s,\tilde{X}_s)$, are symplectic invariant, albeit under the following two \textit{different} realisations of the symplectic Lie algebra given by:
	\begin{align*}
	\begin{cases}
	X_{jk}=x_j\partial_{x_k}-y_k\partial_{y_j} - (q_k\partial_{q_j}+\frac{1}{2}\delta_{jk})
	& 1\leq j\leq k\leq n
	\\Y_{jk}=x_j\partial_{y_k}+x_k\partial_{y_j} +i\partial_{q_j}\partial_{q_k}
	& j<k=1,\dots,n
	\\ Z_{jk}=y_j\partial_{x_k}+y_k\partial_{x_j} + i q_jq_k
	& j<k=1,\dots,n
	\\Y_{jj}=x_j \partial_{y_j} +\frac{i}{2}\partial_{q_j}^2
	& j=1,\dots,n
	\\Z_{jj}=y_j\partial_{x_j}+\frac{i}{2} q_j^2
	&  j=1,\dots,n
	\end{cases}
	\end{align*} 
and 
\begin{align*}
\begin{cases}
\tilde{X}_{jk}=x_j\partial_{x_k}-y_k\partial_{y_j} + q_k\partial_{q_j}+\frac{1}{2}\delta_{jk}
& 1\leq j\leq k\leq n
\\\tilde{Y}_{jk}=x_j\partial_{y_k}+x_k\partial_{y_j} -iq_jq_k
& j<k=1,\dots,n
\\ \tilde{Z}_{jk}=y_j\partial_{x_k}+y_k\partial_{x_j} -i \partial_{q_j}\partial_{q_k}
& j<k=1,\dots,n
\\\tilde{Y}_{jj}=x_j \partial_{y_j} -\frac{i}{2}q_j^2
& j=1,\dots,n
\\\tilde{Z}_{jj}=y_j\partial_{x_j}-\frac{i}{2} \partial_{q_j}^2& j=1,\dots,n
\end{cases}
\end{align*}
Of course, it is not very useful that $D_s$ and $\widetilde{D}_s$ are invariant under different (yet isomorphic) $\mathfrak{sp}(2n,\mathbb{R})$-realisations. Therefore, we will perform a symmetry reduction so that both operators become invariant under \textit{one and the same} Lie algebra. To that end, we need to find the symplectic matrices which commute with the complex structure. 
\begin{lemma} We have that
	 $$
	\mathsf{Sp}_{\mathbb{J}}(2n,\mathbb{R}):= \{M\in\mathsf{Sp}(2n,\mathbb{R})\mid M\mathbb{J} = \mathbb{J}M\}
	$$ 
	defines a realisation for the unitary Lie group.
\end{lemma}
\begin{proof}
 In order to see this, 
assume that $M$ is of the block-form:
$\left(\begin{smallmatrix}
A	&	B\\
C	&	D
\end{smallmatrix}\right),$
where $A,B,C$ and $D$ are $(n\times n)-$matrices. The condition that $M$ is symplectic is equivalent to the one of the following conditions: the matrices $A^TC$ and $B^TD$ are symmetric and $A^TD-C^TB=I$.
So, in order to determine $\operatorname{Sp}_{\mathbb{J}}(2n,\mathbb{R})$ we need to determine the symplectic matrices $M$ which commute with the complex structure $\mathbb{J}$. The latter conditions means that \begin{align*}
M\mathbb{J}=\mathbb{J}M&\iff \mathbb{J}^{-1}M\mathbb{J}=M
\\&\iff
\begin{pmatrix}
0	&	I\\
-I	&	0
\end{pmatrix}\begin{pmatrix}
A	&	B\\
C	&	D
\end{pmatrix}
\begin{pmatrix}
0	&	-I\\
I	&	0
\end{pmatrix} = \begin{pmatrix}
A	&	B\\
C	&	D
\end{pmatrix}	
\\ &\iff\begin{pmatrix}
D	&	-C\\
-B	&	A
\end{pmatrix} = \begin{pmatrix}
A	&	B\\
C	&	D
\end{pmatrix}
\end{align*}
This implies $A=D$ and $B=-C$. In other words the matrix $M$ is of the form:
$M=\left(\begin{smallmatrix}
A	&	B\\ -B	& A
\end{smallmatrix}\right).$
Next, we still have the condition that $M$ is symplectic, i.e. \begin{align*}
&M^T\Omega M=\Omega\\&\iff
\begin{pmatrix}
A^T	&	C^T\\
B^T	&	D^T
\end{pmatrix} \begin{pmatrix}
0	&	I \\
-I	&	0
\end{pmatrix}
\begin{pmatrix}
A	&	B\\
C	&	D
\end{pmatrix} = \begin{pmatrix}
0	&	I \\
-I	&	0
\end{pmatrix}
\\ &\iff \begin{pmatrix}
-C^T& A^T\\
-D^T&B^T
\end{pmatrix}\begin{pmatrix}
A	&	B\\
C	&	D
\end{pmatrix} =\begin{pmatrix}
-C^TA+A^TC&-C^TB+A^TD\\ -D^TA+B^TC&-D^TB+B^TD
\end{pmatrix}=\begin{pmatrix}
0&I\\-I&0
\end{pmatrix}
\\&\iff \begin{cases}
A^TC=C^TA\\
A^TD-C^TB=I\\
B^TC-D^TA=-I\\
B^TD=D^TB
\end{cases}
\end{align*}
This means that $A^TC$ and $B^TD$ should be symmetric matrices and $A^TD-C^TB=I$. But now, due to the first condition this reduces to  $ B^TA=A^TB$ and $A^TA+B^TB=I$. In other words, the matrices we are looking for must be of the form $M=\left(\begin{smallmatrix}
A	&	B\\
-B	&	A
\end{smallmatrix}\right)$
with $B^TA=A^TB$ and $A^TA+B^TB=I$, i.e.
\begin{align*}
\begin{pmatrix}
A^T	&	-B^T\\ B^T	& A^T
\end{pmatrix}\begin{pmatrix}
A	&	B\\ -B	& A
\end{pmatrix}=\begin{pmatrix}
A^TA+B^TB&A^TB-B^TA\\ B^TA-A^TB&B^TB+A^TA
\end{pmatrix}=\begin{pmatrix}
I&0\\0&I
\end{pmatrix}
\end{align*}
Which is exactly the condition for a unitary matrix. The map $\Phi: \mathsf{Sp}_{\mathbb{J}}(2n,\mathbb{R})\to \mathsf{U}(n):M\mapsto A+iB$ gives the wanted isomorphism.
\end{proof}
\subsection{Unitary invariant symplectic Dirac operators}
One can now check that the symplectic Dirac operator and its twist, are unitary invariant differential operators.  This can be done by verifying that the operators commute with the following realisation of unitary Lie algebra $\mathfrak{u}(n)$:
\begin{lemma}
	We have the following operator realisation of the Lie algebra $\mathfrak{u}(n)$:
\begin{align*}
\begin{cases}
A_{jk}=y_j\partial_{x_k}+y_k\partial_{x_j} -x_j\partial_{y_k}-x_k\partial_{y_j}+ i (q_jq_k -\partial_{q_j}\partial_{q_k})&\quad 1\leq j<k\leq n
\\B_{jj}=y_j\partial_{x_j}-x_j\partial_{y_j}+\frac{i}{2}\left(q_j^2-\partial_{q_j}^2\right)
&\quad 1\leq j\leq n
\\ C_{jk}=x_j\partial_{x_k} - x_k\partial_{x_j}+ y_j\partial_{y_k} - y_k\partial_{y_j} + q_j\partial_{q_k} - q_k\partial_{q_j}
&\quad 1\leq j<k\leq n
\end{cases}
\end{align*}
\end{lemma}

\noindent This means that we can refine the $\mathfrak{sp}(2n)$-invariant PDE $D_sf=0$ into two $\mathfrak{u}(n)$-invariant PDEs given by $D_sf = 0$ and $\widetilde{D}_sf=0$, for a symplectic spinor valued polynomial $f \in \mathcal{P}(\mathbb{R}^{2n},\mathbb{C})\otimes \mathcal{S}(\mathbb{R}^n)$. In analogy with the orthogonal case, we call the solutions {\em hermitian symplectic monogenics} (or $h$-symplectic monogenics in short).
\subsection{Symplectic Dolbeault operators}
Moreover, there is a second way of introducing the twist of the symplectic Dirac operator. Let us define the following operators which are known in the literature as the \textit{symplectic Dolbeault operators} \cite{Dol} defined by means of $$D_z=\frac{D_s+i\widetilde{D}_s}{2}$$ and $$ D_z^{\dagger}:=\frac{D_s-i\widetilde{D}_s}{2}.$$
One easily verifies that
\begin{align*}
\frac{1}{2}(D_s+i\widetilde{D}_s)&=-\sum_{j=1}^n\mathfrak{F}_j\partial_{z_j}\quad\\\ \frac{1}{2}(D_s-i\widetilde{D}_s)&=\sum_{j=1}^n\mathfrak{F}_j^{\dagger}\partial_{\overline{z}_j}.
\end{align*}
where we have introduced the symbols $\mathfrak{F}_j=(q_j+\partial_{q_j})$ and $\mathfrak{F}_j^{\dagger}=(q_j-\partial_{q_j})$.
The structure of the operators $D_z$ and $D_z^{\dagger}$ is similar to the orthogonal case. However, the raising/lowering operators $\mathfrak{F}_j$ and $\mathfrak{F}_j^{\dagger}$ are used instead of isotropic Witt vectors $\mathfrak{f}_j$ and $\mathfrak{f}_j^{\dagger}$ (see for instance \cite{spin} and the references therein).
\subsection{Class of simultaneous solutions of $D_s$ and $\widetilde{D}_s$}
We will now describe a wide class of examples of $h$-symplectic monogenics, by making the link with holomorphic functions in several variables.
Let $f:\Omega\subset \mathbb{C}^n\to \mathbb{C}$ be a complex-valued function in several complex variables which is of the class $\mathcal{C}^1(\Omega)$ (i.e.\ continuously differentiable). We say that $f$ is holomorphic (in several variables) if $\partial_{\overline{z}_j}f(z_1,\dots,z_n)=0$ for all $1\leq j \leq n$,
where $\partial_{\overline{z}_j}:=\frac{1}{2}(\partial_{x_j}+\partial_{y_j})$ is the Cauchy-Riemann operator in the relevant variable. Moreover, we denote the set of holomorphic functions in $\Omega$ by $\operatorname{Hol}(\Omega)$.
\par In order not to overload notations, we use the summation convention.
Suppose that we have a function of the form $F({x},{y},{q})=e^{-\frac{1}{2}|{q}|^2}H({x},{y})$. Letting the symplectic Dirac operator act on $F$ gives:
\begin{align*}
D_s\left(e^{-\frac{1}{2}|{q}|^2}H({x},{y})\right)
&=(iq_k\partial_{y_k}-\partial_{x_k}\partial_{q_k})\left(e^{-\frac{1}{2}|{q}|^2}H({x},{y})\right)
\\&=iq_ke^{-\frac{1}{2}|{q}|^2}\partial_{y_k}H({x},{y})+e^{-\frac{1}{2}|{q}|^2}q_k\partial_{x_k}H({x},{y})
\\&=e^{-\frac{1}{2}|{q}|^2}q_k(\partial_{x_k}+i\partial_{y_k})H({x},{y})
\end{align*}
We note that this equals zero if $(\partial_{x_k}+i\partial_{y_k})H({x},{y})=0$ for all $k=1,\dots,n$, i.e. if $H({x},{y})$ is a holomorphic function in several variables. Completely similar, 
\begin{align*}
\widetilde{D_s}\left(e^{-\frac{1}{2}|{q}|^2}H({x},{y})\right)
&=(iq_k\partial_{x_k}+\partial_{y_k}\partial_{q_k})\left(e^{-\frac{1}{2}|{q}|^2}H({x},{y})\right)
\\&=e^{-\frac{1}{2}|{q}|^2}q_k(i\partial_{x_k}-\partial_{y_k})H({x},{y})
\\&=ie^{-\frac{1}{2}|{q}|^2}q_k(\partial_{x_k}+i\partial_{y_k})H({x},{y}),
\end{align*}
which is zero for holomorphic $H({x},{y})$. \par This means that every function of the form $	e^{-\frac{1}{2}(q_1^2+\dots+q_n^2)}H({x},{y})$ with $H$ an holomorphic function is several variables is a solution of both $D_s$ and $\widetilde{D}_s$.
This observation generalises the class of solutions obtained by Habermann in \cite{Habermann} for $n=1$. As a matter of fact, it turned out that there are much \textit{more solutions} than the ones of this form. In order to describe these systematically, we will need the notion of Howe dualities and corresponding Fischer decompositions. This will be done in full detail in our upcoming paper \cite{sHow}. In the following section we will reveal the algebraic structures required for this approach.
\section{A unitary Howe duality associated with $D_s$ and $\widetilde{D}_s$}
\noindent Recall that the Lie algebra $\mathfrak{su}(1,2)$ is a quasi-split real form of the complex Lie algebra $\mathfrak{sl}(3)$ and is defined in terms of matrices as 
\begin{align*}
\mathfrak{su}(1,2) = \left\{\begin{pmatrix}
\alpha	&	\beta	&	ic \\
\gamma	&	\overline{\alpha} - \alpha &	-\overline{\beta}\\
id		&	-\overline{\gamma}	& - \overline{\alpha}.
\end{pmatrix}\mid c,d\in \mathbb{R}\ \& \ \alpha,\beta,\gamma\in\mathbb{C}\right\}
\end{align*}
Some calculations now lead to observation that the Lie algebra generated by the symplectic Dirac operators $D_s,\widetilde{D}_s$ and their duals $X_s,\widetilde{X_s}$ gives rise to a copy of the Lie algebra $\mathfrak{su}(1,2)$. In order to close the algebra, we introduce the following differential operators:
\begin{align*}
\mathcal{O}:=\sum_{j=1}^ni(x_j\partial_{y_j}
-y_j\partial_{x_j})+\partial_{q_j}^2-q_j^2,
&& \Delta:=\sum_{j=1}^n\partial_{x_j}^2+\partial_{y_j}^2&&\text{ and }
&& r^2:=\sum_{j=1}^n {x_j}^2+{y_j}^2.
\end{align*}
Together with symplectic Dirac operators $D_s,\widetilde{D}_s$ and their duals $X_s,\widetilde{X}_s$ these operators satisfy the commutator relations of $\mathfrak{su}(1,2)$. Moreover, their commutators are given by:
$$\begin{tabular}{c| c c c c c c c c }
$[\cdot,\cdot]$	& $D_s$ & $\widetilde{D}_s$  & $\Delta$  & $X_s$  & $\mathbb{E}$  & $\mathcal{O}$ & $\widetilde{X}_s$ & $r^2$  \\
\hline
$D_s$ & 0 & $\Delta$ & 0  & -$\mathbb{E}$  & $D_s$ & $-3\widetilde{D}_s$ & $-\mathcal{O}$  &  $-2\widetilde{X}_s$\\

$\widetilde{D}_s$ & $-\Delta$& 0 & 0 & $\mathcal{O}$ & $\widetilde{D}_s$ & $3D_s$ & $-\mathbb{E}$ & $2X_s$ \\

$\Delta$ & 0 & 0 & 0 & 	$\widetilde{D}_s$ & $\Delta$ & 0 & 	$-2{D}_s$  & $\mathbb{E}$ \\

$X_s$ & $\mathbb{E}$ & $-\mathcal{O}$ & $-\widetilde{D}_s$ & 0 & $X_s$ & $-3\widetilde{X}_s$ & $-r^2$ & 0 \\

$\mathbb{E}$ & $-D_s$ & $-\widetilde{D}_s$ & $-\Delta$ & $-X_s$ & 0 & 0 & $-\widetilde{X}_s$  & 0  \\

$\mathcal{O}$ & $3\widetilde{D}_s$ & $-3D_s$  & 0  & $3\widetilde{X}_s$ & 0 & 0 & $-3X_s$ & 0 \\

$\widetilde{X}_s$ & $\mathcal{O}$ & $\mathbb{E}$ & $2D_s$ & $r^2$ & $\widetilde{X}_s$ & $-3X_s$ & 0 & 0 \\

$r^2$ & $2\widetilde{X}_s$  & $-2X_s$ & $-\mathbb{E}$ & 0 & $-2r^2$ & 0 & 0 & 0 \\

\end{tabular}$$

\begin{enumerate}
\item We have two copies of the Heisenberg algebra: $$\operatorname{Alg}\{D_s,\tilde{D}_s,\Delta\} \cong \operatorname{Alg}\{X_s,\widetilde{X}_s,r^2\} \cong \mathfrak{h}_3.$$
\item We have two copies of the Lie algebra $\mathfrak{sl}(2)$
\[\operatorname{Alg}\{D,D^{\dagger},\mathbb{E}\} \cong \operatorname{Alg}\{\tilde{D},\tilde{D}^{\dagger},\mathbb{E}\} \cong \operatorname{Alg}\{\Delta,N,\mathbb{E}\}\cong \mathfrak{sl}(2).\]
\end{enumerate}
This means that there is a canonical $\mathfrak{su}(1,2)$-action on the space of spinor valued polynomials $\mathcal{P}(\mathbb{R}^{2n},\mathbb{C})\otimes\mathcal{S}(\mathbb{R}^n)$ where restricting to the subalgebra $\operatorname{Alg}(D_s,X_s)$ corresponds to $\mathfrak{sl}(2)$-copy obtained in the Howe duality for symplectic Clifford analysis (see \cite{sym} for more details). Now, taking into account the symplectic Dirac operators and its twists, we obtain the dual pair $\mathsf{U}(n)\times \mathfrak{su}(1,2)$ (i.e.\ the underlying group of invariance, together with the algebra generated by the operators and their duals). \par
We now focus on the reduction of the symplectic spinor space. 
In the orthogonal case, the spinor space $\mathbb{S}$ decomposes under the action of the unitary group $\mathsf{U}(n)$ as  $\mathbb{S}=\bigoplus_r\mathbb{S}_{(r)},$
with $\mathbb{S}_{(r)}$ inequivalent irreducible pieces, which are eigenspaces of the \textit{fermionic} quantum oscillator (also called spin-Euler operator, see for instance \cite{spin}). In the symplectic case, the relevant operator for decomposing the infinite dimensional spinor space is the \textit{bosonic} quantum oscillator. The hamiltonian of the quantum oscillator, the so-called Hermite operator, is given by \begin{align*}
\mathcal{H}:\mathcal{S}(\mathbb{R}^n)\to \mathcal{S}(\mathbb{R}^n), \quad f({q})\mapsto \frac{1}{2}\sum_{j=1}^n(\partial_{q_j}^2-q_j^2)f({q}).
\end{align*} 
Note that we can write $$\mathcal{O}=\sum_{j=1}^ni(x_j\partial_{y_j}
-y_j\partial_{x_j})+2\mathcal{H},$$ so that the Hermite operator is in fact the spinor-valued part of the operator $\mathcal{O}$, i.e.\ the differential operator in $\partial_{q_j}$ and the variables $q_j$. Moreover, the eigenspaces can be identified with the irreducible decomposition of $\mathbb{S}^{\infty}$ into $\mathfrak{u}(n)$-irreducible representations. This means that the symplectic spinor space $\mathbb{S}^{\infty}$ decomposes into $\mathfrak{u}(n)$-irreducible representations ${\widetilde{S}}^{\infty}_{(k)}$ of dimension ${n+k-1 \choose k}$ which can be thought of as $k$-homogeneous polynomials or the eigenspaces of the Hermite operator $\mathcal{H}$.
\par 
Moreover, the solutions of the corresponding Dirac operators, called monogenics, can be introduced from a purely representation theoretical viewpoint. In general, this boils down to determining the decomposition (this is called a \textit{Fischer decomposition})
$
\mathcal{P}_k(\mathbb{R}^m,\mathbb{C})\otimes \mathbf{S}
$
where $\mathbf{S}$ is the spinor space, which is $\mathbb{S}$ in the orthogonal case and $\mathbb{S}^{\infty}$ in the symplectic case, where we take $m=2n$ in particular. Moreover, the space of $k$-homogeneous polynomials $\mathcal{P}_k$ coincides with the $k$-symmetric power of the fundamental representation of resp.\ the orthogonal or symplectic algebra. We denote by $\mathcal{M}_k$ the $k$-homogeneous solutions of the Dirac operator $\partial_x$, these are called monogenics. They can be defined as follows:
\begin{align*}
\mathcal{M}_k
\leftrightarrow 	(k,0,\dots,0)\boxtimes \mathbb{S}= 	(k)\boxtimes \left(\frac{1}{2},\dots,\frac{1}{2}\right)\cong \left(k+\frac{1}{2},\dots,\frac{1}{2}\right), \end{align*}
where $\boxtimes$ denotes the Cartan product of the $\mathfrak{so}(m)$-representations.
In the symplectic case, we analoguously obtain:
\begin{align*}
\mathcal{M}_k^s
\leftrightarrow 	(k,0,\dots,0)_s\boxtimes \mathbb{S}^{\infty}&= 	(k)_s\boxtimes \left(\left(-\frac{1}{2},\dots,-\frac{1}{2}\right)\oplus \left(-\frac{1}{2},\dots,-\frac{3}{2}\right)\right)\\&\cong \left(k-\frac{1}{2},\dots,-\frac{1}{2}\right)\oplus \left(k-\frac{1}{2},\dots,-\frac{3}{2}\right). \end{align*} 
In order to obtain an algebraic characterisation of the space of $h$-symplectic monogenics, one proceeds as follows. First of all, we note that we need to consider the symplectic spinors $\mathbb{S}^{\infty}$ from an unitary viewpoint. We saw that $\mathbb{S}^{\infty}$ decomposes as an infinite direct sum of finite dimensional $\mathfrak{u}(n)$-modules  ${\widetilde{S}}^{\infty}_{(k)}$ which are in fact eigenspaces of the Hermite operator.
We denote the branched spinor space (which is in fact a direct sum of $\mathfrak{u}(n)$-irreps) by $\widetilde{\mathbb{S}^{\infty}}$. However, the space of $k$-homogeneous polynomials is \textit{not} irreducible as a $\mathfrak{u}(n)$-module and we denote the branched module by $\widetilde{(k)}$. This means that we are left with the following Cartan product $
\mathcal{M}_k^{hs}
\leftrightarrow 	\widetilde{(k)}\boxtimes \widetilde{\mathbb{S}^{\infty}}
$ as a representation theoretical definition of the $h$-symplectic monogenics. Recall that this are the symplectic spinor-valued polynomial functions $f\in\mathcal{P}(\mathbb{R}^{2n},\mathbb{C})\otimes \mathcal{S}(\mathbb{R}^n)$ that satisfy the system of unitary unitary-invariant partial differential equations
$$\begin{cases}
D_sf&=0\\
\widetilde{D}_sf&=0
\end{cases}$$ The explicit calculation of the Cartan product (and more generally the tensor product) will be done in \cite{sHow}. Moreover, as an application we will prove a Fischer decomposition for the Howe dual pair we obtained in this paper.
\section{Conclusion}
In this paper we investigated a new Howe dual pair occurring in symplectic Clifford analysis by allowing a compatible complex structure. This Howe duality is of the form $(G,\mathfrak{g}')$ where $G$ is the underlying  invariance group for which the relevant Dirac operators are invariant and $\mathfrak{g}'$ is the algebra generated by the Dirac operators and their duals. Depending on the orthogonal or symplectic framework, we have the following `types' of Clifford analysis and \textit{refinements} thereof:
\begin{enumerate}
\item Orthogonal geometry (giving rise to a Clifford algebra) \begin{enumerate}
	\item Clifford analysis:  $\mathsf{SO}(m)\times \mathfrak{osp}(1|2)$
	\item Hermitian Clifford analysis  $\mathsf{U}(m)\times \mathfrak{osp}(2|2)$
	\item Quaternionic Clifford analysis  $\mathsf{USp}(m)\times \mathfrak{osp}(4|2)$
\end{enumerate}
\item Symplectic geometry (giving rise to a Weyl algebra)
\begin{enumerate}
	\item Symplectic Clifford analysis:  $\mathsf{Sp}(2m)\times \mathfrak{sl}(2)$
	\item Hermitian symplectic Clifford analysis:  $\mathsf{U}(m)\times \mathfrak{su}(1,2)$
	\item Quaternionic symplectic Clifford analysis: $\mathsf{USp}(m)\times $?
\end{enumerate}
\end{enumerate}
Thus far, we extended the framework of hermitian Clifford analysis in the presence of a symplectic structure in the case of the (flat) K\"ahler manifold $\mathbb{R}^{2n}$. It is an interesting question to further reduce the symmetry to the compact symplectic group $\mathsf{USp(n)}$ so that we have the chain $\mathsf{Sp}(2n)\supset \mathsf{U}(n)\supset \mathsf{USp}(n)$. In our furture work \cite{sHow}, we will describe the Fischer decomposition accompanying this new Howe dual pair. 

\subsection*{Funding information}
The author is supported by the FWO-EoS project ‘Symplectic Techniques in Differential Geometry’ G0H4518N.  



\bibliographystyle{SciPost_bibstyle}

\end{document}